\documentclass{amsproc}

\usepackage{amsmath}
\usepackage{amsthm}
\usepackage{amssymb}
\usepackage{graphicx}
\usepackage{multirow}
\usepackage{mathtools}
\usepackage{tikz-cd}
\usepackage{tikz}
\usepackage{enumerate}
\usepackage{mathdots}
\usepackage{enumitem}
\usepackage{array}
\usepackage{hyperref}
\usepackage{dynkin-diagrams}

\newtheorem{theorem}{Theorem}[section]
\newtheorem{corollary}[theorem]{Corollary}
\newtheorem{lemma}[theorem]{Lemma}
\newtheorem{proposition}[theorem]{Proposition}
\newtheorem{remark}[theorem]{Remark}
\newtheorem{example}[theorem]{Example}

\newtheorem{setup}[theorem]{Setup}

\newtheorem{question}[theorem]{Question}
\newtheorem*{theorem*}{Theorem}

\newtheorem{definition}[theorem]{Definition}

\def\Aut{\operatorname{Aut}}

\def\Pic{\operatorname{Pic}}

\def\PGL{\operatorname{PGL}}

\def\res{\operatorname{res}}
\def\Lie{\operatorname{Lie}}

\def\Stab{\operatorname{Stab}}

\def\C{{\mathbb C}}

\def\G{{\mathbb G}}

\def\P{{\mathbb P}}
\def\Q{{\mathbb Q}}
\def\R{{\mathbb R}}
\def\Z{{\mathbb Z}}

\def\cD{{\mathcal D}}

\def\cO{{\mathcal O}}

\def\cR{{\mathcal R}}
\def\cS{{\mathcal S}}

\def\cU{{\mathcal U}}

\title[Automorphism groups of toroidal horospherical varieties]{Automorphism groups of toroidal horospherical varieties}

\author[L. Barban]{Lorenzo Barban}
\address{Center for Complex Geometry, Institute for Basic Science (IBS), 55 Expo-ro, Yuseong-gu, Daejeon, 34126, Republic of Korea}
\email{lorenzobarban@ibs.re.kr}

\author[D. Hwang]{DongSeon Hwang}
\address{Center for Complex Geometry, Institute for Basic Science (IBS), 55 Expo-ro, Yuseong-gu, Daejeon, 34126, Republic of Korea}
\email{dshwang@ibs.re.kr}

\author[M. Kwon]{Minseong Kwon}
\address{Morningside Center of Mathematics, Academy of Mathematics and Systems Science, Chinese Academy of Sciences, Beijing, 100190, China}
\email{minseong@amss.ac.cn}


\begin{document}
\begin{abstract}
	We establish a structure theorem for the connected automorphism groups of smooth complete toroidal horospherical varieties, that is, toric fibrations over rational homogeneous spaces. The key ingredient is a characterization of the Demazure roots of toric fibers that extend to the total space.
    In particular, we provide a criterion for the reductivity of the connected automorphism groups of such varieties. As an application, we prove the K-unstability of certain $\mathbb{P}^1$-bundles over rational homogeneous spaces.
\end{abstract}  

\maketitle

\section{Introduction}
The automorphism group of an algebraic variety encodes rich information about its geometry. For a smooth complete rationally connected variety $X$ over the field $\mathbb{C}$ of complex numbers, the identity component $\mathrm{Aut}^0(X)$ is a linear algebraic group. Finding a criterion for the reductivity of  $\mathrm{Aut}^0(X)$ has received considerable  attention, partly  motivated by the Matsushima--Lichnerowicz theorem (\cite{Matsushima, Lichnerowicz}) which asserts that if a  compact K\"ahler manifold $X$ admits a constant scalar curvature metric (in particular, a K\"ahler--Einstein metric when $X$ is Fano), then $\mathrm{Aut}^0(X)$ is reductive.    

The structure of $\mathrm{Aut}^0(X)$ is well understood for several important classes of spherical varieties, which are  rationally connected varieties endowed with rich group actions. For example, if $X$ is a rational homogeneous space the automorphism group is completely determined (\cite{Onishchik, DemazureRHS}): in particular,  $\mathrm{Aut}^0(X)$ is semisimple, and hence reductive. Wonderful varieties can be regarded as generalizations of rational homogeneous spaces, and $\mathrm{Aut}^0(X)$ for a wonderful variety $X$ is also semisimple (\cite{PezziniWonderful}). 

In contrast, toric varieties form another well-known class of spherical varieties whose connected automorphism groups exhibit a rather different behavior. For a complete toric variety $X$,  $\mathrm{Aut}^0(X)$ is in general neither semisimple nor reductive, and its structure can be described explicitly in terms of Demazure roots (\cite{demazure}).  More precisely, the dimension of $\mathrm{Aut}^0(X)$ can be computed as the sum of the dimension of the variety and the number of Demazure roots:
\[ \dim(\Aut^0(X)) = \dim(X) + |\cS(\Sigma)| + |\cU(\Sigma)| \]
where $\cS(\Sigma)$ and $\cU(\Sigma)$ denote respectively the sets of semisimple and unipotent Demazure roots of the fan $\Sigma$ associated with $X$. In particular, $\mathrm{Aut}^0(X)$ is reductive if and only if $\cU(\Sigma) = \emptyset$. See Theorem \ref{theorem:liendo} (2) and (4).

Horospherical varieties occupy an intermediate position between these classes: they are birational to toric bundles over rational homogeneous spaces, thereby simultaneously generalize both rational homogeneous spaces and toric varieties.
For a smooth complete horospherical variety $X$, $\Aut^{0}(X)$ has been investigated in various contexts.
For instance, when the Picard number of $X$ is at most two, $\Aut^{0}(X)$ is studied as a consequence of the classification of such an $X$, see \cite{PasquierAut} and \cite{villeneuve}.

Recall that a horospherical variety is in particular a spherical variety, i.e., a normal $G$-variety containing an open dense $B$-orbit where $G$ is a connected reductive linear algebraic group and $B$ is a Borel subgroup of $G$.
The connected automorphism group $\Aut^{0}(X)$ of a spherical variety $X$ has been extensively studied in the literature.
On the one hand, if $X$ is smooth, complete and toroidal, meaning that every $B$-invariant divisor is $G$-invariant, then $\Aut^{0}(X)$ is studied after \cite{BienBrion}. In this case, a description of Levi subgroups, i.e., maximal connected reductive subgroups of $\Aut^{0}(X)$, is given in \cite{PezziniReductive}.
On the other hand, $B$-normalized $\G_{a}$-subgroups of the automorphism group and their action on (quasi-)affine spherical varieties have been systematically studied by \cite{ArzhantsevRoot, AZ2, AZ3}. There has been a recent progress to generalize this approach on arbitrary spherical varieties (cf. \cite{AZ}). Still, a complete characterization of the reductivity of the automorphism group of spherical varieties is missing.

In this paper we study the structure of $\mathrm{Aut}^0(X)$ for a smooth complete toroidal horospherical variety $X$ with a global approach. Let us recall that a complete toroidal horospherical variety is a complete toric bundle $\Phi\colon X\to G/P$ extending the fiber bundle $\varphi\colon G/H\to G/P$ with an algebraic torus $S$ as a fiber. 
Our approach characterizes the Demazure roots of the toric fibers that induce $\mathbb{G}_a$-actions on the total space. As in the toric case, we introduce the sets $\cS^{+}_{G}(X)$ and $\cU^{+}_{G}(X)$ of semisimple and unipotent Demazure roots for $X$, respectively, which allow us to compute the dimension of $\mathrm{Aut}^0(X)$ and to characterize its reductivity in terms of the generalized Demazure roots.

\begin{theorem}[Corollary \ref{coro:LieAut(X)} and \ref{coro:reductive}]\label{main1}
Let $X$ be a smooth complete toroidal horospherical variety,  that is, a complete toric bundle over a rational homogeneous space $G/P$ with a fiber $F$. Then the following hold.
    \begin{enumerate}
        \item Denote by $V(m)$ the irreducible $G$-module of highest weight $m$ with respect to a fixed Borel subgroup $B^+$. Then
        \[\dim \Aut^{0}(X) =\dim \Aut^{0}(G/P) + \dim F + |\cS^{+}_{G}(X)| + \sum_{m \in \cU^{+}_{G}(X)} \dim V(m).\] 
        \item $\Aut^{0}(X)$ is reductive if and only if $\cU^{+}_{G}(X) = \emptyset$.
    \end{enumerate} 
\end{theorem}

A key step toward Theorem \ref{main1} is the description of $\mathbb{G}_a$-actions on $X$ in terms of two combinatorial data: the Demazure roots of the toric fiber $F$ and the color map 
\[
\epsilon^{+}: \{\text{prime $B^{+}$-divisors on $G/H$}\} \longrightarrow N_{S} \coloneqq (M_{S})^{\vee}.
\]
See \S \ref{subsection:horospherical} for the precise definition. This leads to the notion of $B^{+}$-roots of $X$ (Definition \ref{definition:horospherical roots}), which generalizes Demazure roots to the horospherical setting. 

To describe the extending condition for $\G_a$-actions on the toric variety $F$ to $X$, we use the notion of an $L$-normalized $\G_a$-action, that is, a $\G_a$-action compatible with the $L$-action up to a character of $L$, equivalently, the conjugation by $L$ rescales the $\G_a$-parameter by that character. Denote by $\cR^{+}_{G}(X)$ the set of Demazure roots for $X$, i.e., the union of $\cS^{+}_{G}(X)$ and $\cU^{+}_{G}(X)$.

\begin{theorem}[Theorem \ref{theorem:B-normalized Ga action}]\label{mainthm:lifting}
For each Demazure root $m$ of the $S$-toric variety $F$,  the $S$-normalized $\G_{a}$-action on $F$ associated to $m$ extends to a $B^{+}$-normalized $\G_{a}$-action on $X$ preserving every fiber of $\Phi$ if and only if $m \in \cR^{+}_{G}(X)$.
If this holds, then the $B^{+}$-normalized extension is unique.
\end{theorem}

Let us emphasize that, in general, $B^{+}$-normalized $\G_{a}$-actions on horospherical varieties are not uniquely determined by their weights (cf. \cite[\S~5.5]{ArzhantsevRoot}).

Now Theorem \ref{main1} can be obtained, together with Theorem \ref{mainthm:lifting}, from the following structure theorem for $\mathrm{Aut}^0(X)$ describing the Levi decomposition of the Lie group $\mathrm{Aut}^0(X)$. For $m \in \cR^{+}_{G}(X)$, we denote by $U^{+}_{m}$ the $\G_{a}$-subgroup of $\Aut^{0}(X)$ associated to the $B^{+}$-normalized $\G_{a}$-action on $X$.

\begin{theorem}[Theorem \ref{thm:structure of K}      (\ref{Kstructure1}), and Corollary \ref{coro:reductive} (\ref{item:coro:reductive2})]\label{mainthm:Levi decomp} $ $
    \begin{enumerate}    
    \item 
    The unipotent radical $R^{u}(\Aut^{0}(X))$ is generated by $U^{+}_{m}$ for $m \in \cU^{+}_{G}(X)$ and their $G$-conjugates. As a $G$-module, $\Lie(R^{u}(\Aut^{0}(X)))$ is the multiplicity-free direct sum of irreducible $G$-modules with $B^{+}$-highest weights $m \in \cU^{+}_{G}(X)$.

    \item \label{item:mainthm3}
    A connected reductive subgroup $L$ of $\Aut^{0}(X)$ is generated by $G$, $\Aut_{G}(X)$ and $U^{+}_{m}$ for $m \in \cS^{+}_{G}(X)$. As a $G$-module, $\Lie (L)$ is the multiplicity-free direct sum of the Lie algebra of the closed subgroup of $\Aut^{0}(X)$ generated by $G$ and $\Aut_{G}(X)$, and 1-dimensional $G$-modules with weights $m \in \cS^{+}_{G}(X)$.
    Moreover, $L$ is a Levi subgroup of $\Aut^{0}(X)$, provided that the natural morphism $G \rightarrow \Aut^{0}(G/P)$ is surjective.
\end{enumerate}
\end{theorem}

By the Lie group-Lie algebra correspondence, one can deduce the structure of $\Aut^{0}(X)$ from Theorem~\ref{mainthm:Levi decomp}, namely Theorem~\ref{main1}.
In Theorem \ref{mainthm:Levi decomp} (\ref{item:mainthm3}), we remark that the surjectivity of the morphism $G \rightarrow \Aut^{0}(G/P)$ is not a restrictive condition.
Indeed, this condition always holds by replacing $G$ by the bigger reductive subgroup $\Aut^{0}(X,\, \partial_{G}X)$ of $\Aut^{0}(X)$, called the group of \emph{completely regular automorphisms} (see \cite{brionWonderful}), if necessary (cf. Remark \ref{rmk:G->Aut often surjects}).
Thanks to this, we show the following:

\begin{corollary}[Corollary \ref{corollary:LeviSubgroup}]\label{corollary:levi}
    A Levi subgroup of $\Aut^{0}(X)$ is generated by $\Aut^{0}(X,\, \partial_{G}X)$ and $U^{+}_{m}$ for $m \in \cS^{+}_{G}(X)$.    
\end{corollary}

As Theorem \ref{main1} provides a useful criterion for detecting non-reductive automorphism groups, it also yields a method to construct new examples of K-unstable Fano varieties, via the Matsushima obstruction (\cite{Matsushima, KStabilityReductive}) and a result of Delcroix that the K-polystability is equivalent to the K-semistability for horospherical varieties (\cite[Corollary 5.7]{DelcroixENS2020}). In particular, applying Theorem \ref{mainthm:Levi decomp}, we study the automorphism groups of projective bundles of the form $\mathbb{P}_Y\Big(\bigoplus_i L_i\Big) \rightarrow Y,$ 
where $Y$ is a rational homogeneous space and each $L_i$ is a line bundle on $Y$. We show that
\[
\Aut^0\Big(\mathbb{P}_Y\Big(\bigoplus_i L_i\Big)\Big)
\]
is reductive if and only if, for any $i \neq j$ with $L_i \not\simeq L_j$, the line bundle $L_i \otimes L_j^{\vee}$ is not nef (see Theorem \ref{theorem:application}), providing a concrete and effective criterion for the reductivity of automorphism groups of decomposable projective bundles over rational homogeneous spaces.

As a special case we focus on the K-stability of Fano $\mathbb{P}^1$-bundles over Fano varieties, which has recently received considerable attention. For instance, \cite{ZhangZhou} shows that $\mathbb{P}^1$-bundles of the form $\mathbb{P}_V (L^{\vee} \oplus \cO_V)$  over a Fano variety $V$ are K-unstable, provided that the Fano index of $V$ is at least $2$ and $L = -\frac{1}{r}K_{V}$ for some rational number $r > 1$. The latter technical assumptions are essential for their approach. In contrast, Theorem \ref{theorem:application} immediately produces explicit examples of K-unstable Fano $\mathbb{P}^1$-bundles over rational homogeneous spaces.
\begin{corollary}[Corollary \ref{coro:KUnstableCriterion}]
    Let $Y$ be a rational homogeneous space, and let $L$ be a nontrivial nef line bundle such that $K_Y^{\vee}\otimes L^{\vee}$ is ample. Then $ X\coloneq \P(\cO_Y\oplus L^{\vee})$ is a smooth K-unstable Fano variety.
\end{corollary}

See Example \ref{example:KunstableFanoIndex1} for an explicit example of a K-unstable Fano $\mathbb{P}^1$-bundle over a smooth projective Fano variety that is a rational homogeneous space and has Fano index $1$.

As a final comment, we may ask whether the structure of $\Aut^{0}(X)$ described in Theorem \ref{mainthm:Levi decomp} and Corollary \ref{corollary:levi} holds for a broader class of spherical varieties.
  
\begin{question} \label{question}
	Let $X$ be a smooth complete spherical variety.
	\begin{enumerate}
	\item \label{question:levi} How to construct a Levi subgroup $L$ of $\Aut^{0}(X)$?

	\item \label{question:unipotent} How to describe the unipotent radical $R^{u}(\Aut^{0}(X))$ in terms of the geometry of $X$?
    More precisely, which $\G_{a}$-actions on $X$ generate $R^{u}(\Aut^{0}(X))$?
	\end{enumerate}
\end{question}

Partial answers to Question \ref{question} are known in several cases. For toroidal spherical varieties, reductive subgroups of $\Aut^{0}(X)$ are studied in \cite{PezziniReductive}, where in particular Question \ref{question} (\ref{question:levi}) is settled.
Note, however, that Question \ref{question} (\ref{question:unipotent}) still remains open, except in the cases of toroidal horospherical varieties, which are considered in our Theorems~\ref{main1}-\ref{mainthm:Levi decomp}, and of wonderful varieties (see \cite{PezziniWonderful}), where $\Aut^{0}(X)$ is in fact semisimple.
A complete answer to Question \ref{question} (\ref{question:unipotent}) would imply an effective criterion for the reductivity of $\Aut^{0}(X)$.

We were informed by a private communication with Roman Avdeev that a description of the connected automorphism group of complete toroidal horospherical varieties can be obtained by combining some results presented in \cite{AZ}. While their setup, namely arbitrary spherical varieties, is more general, our approach is more direct for smooth complete toroidal horospherical varieties.

\subsection*{Outline} In Section \ref{section:preliminaries} we recall the necessary background on automorphism groups of complete varieties, with particular attention to the cases of toric varieties (see Theorem \ref{theorem:liendo}) and rational homogeneous spaces (see Theorem \ref{thm:Aut G/P}). We also fix the notation for toroidal horospherical varieties, which constitute the main object of this work (see \S \ref{subsection:ToroidalHorospherical}).

Section \ref{section:maintheorem} forms the technical core of the paper: here we analyze in detail the connected automorphism groups of toroidal horospherical varieties, by using the automorphism groups of the fibers (which are toric varieties) and of the base (which are rational homogeneous spaces) of the natural fibration carried by them (see Lemma \ref{lemma: key lemma}, Theorem \ref{thm:structure of K}).

In Section \ref{section:applications} we apply our work to the special case of totally decomposable projective bundles over rational homogeneous spaces, relating the reductivity of their automorphism groups with the positivity of the line bundles defining such structure (see Theorem \ref{theorem:application}). We conclude by proving the K-unstability of certain Fano $\P^1$-bundles over rational homogeneous spaces.

\subsection*{Notation}\label{ssection:notation}
We work in the category of varieties over the field $\C$ of complex numbers. Here, a variety means an integral separated scheme of finite type over $\C$. For an algebraic group $A$, we denote by  $A^{0}$ the identity component of $A$ and by $\Lie(A)$ its Lie algebra; we refer to \cite{OVbook} for details on Lie theory. 
We denote by $G$ a connected reductive linear algebraic group. 
Given a linear algebraic group $L$, we set $R^{u}(L)$ to be its unipotent radical, and $M_{L}$ to be its character lattice, that is, the group of algebraic group homomorphisms $L \rightarrow \C^{\times}$. We use the geometric projectivization convention for projective bundles over varieties, that is, points correspond to $1$-dimensional linear subspaces.

\subsection*{Acknowledgements} The authors would like to thank Michel Brion, Thibaut Delcroix and Luis Eduardo Sol\'a Conde for fruitful conversations. We would also like to thank Roman Avdeev for the helpful suggestions and references, including \cite{AZ}. Lorenzo Barban was supported by the Institute for Basic Science (IBS-R032-D1). DongSeon Hwang was supported by the National Research Foundation of Korea(NRF) grant funded by the Korea government(MSIT) (2021R1A2C1093787)
and the Institute for Basic Science (IBS-R032-D1). 
Minseong Kwon was partially supported by the National Research Foundation of Korea (NRF) grant funded by the Korea government (MSIT) (RS-2021-NR062093) 
and the Institute for Basic Science (IBS-R032-D1).

\section{Preliminaries}\label{section:preliminaries}

\subsection{Algebraicity and linearity of automorphism groups} \label{subsection:Aut0}
Let $X$ be a complete variety.
By \cite[Theorem (3.7)]{MatOo}, the automorphism group functor of $X$ is representable by a group scheme which is locally of finite type over $\C$.
We denote by $\Aut^{0}(X)$ its identity component.
By construction, $\Aut^{0}(X)$ is a connected algebraic group, and by \cite[Lemma (3.4)]{MatOo}, $\Lie(\Aut^{0}(X))$ is naturally identified with $H^{0}(X,\, T_{X})$ for the tangent sheaf $T_{X}$ of $X$.
Furthermore, if the Picard variety $\Pic^{0}(X)$ is trivial, then by \cite[Remark after Corollary 1]{Ramanujam}, $\Aut^{0}(X)$ is a linear algebraic group.
In particular, if $X$ is normal and rationally connected, then $\Aut^{0}(X)$ is linear algebraic (see \cite[Theorem IV.1.1]{LangAb}).

\subsection{Toric varieties and Demazure roots}\label{subsection:toricDemazure}

We briefly recall the notation for toric varieties, and discuss their automorphism group. We refer to \cite{CLS} for details.

Let $S$ be an algebraic torus of dimension $n$, that is, a linear algebraic group isomorphic to $(\C^{\times})^n$. We denote by $M_S$ (resp. $N_S$) the lattice of characters of $S$ (resp. of $1$-parameter subgroups of $S$), and by $\langle \cdot, \cdot \rangle\colon M_S\times N_S\to \Z$ the associated nondegenerate bilinear pairing. Given $m\in M_S$ (resp. $v\in N_S$), the associated character (resp. $1$-parameter subgroup) will be denoted by $\chi^m\colon S\to \C^\times$ (resp. $\lambda_v\colon \C^\times\to S$).

A \emph{toric variety} $X$ is a normal variety, containing $S$ as an open subset and such that the left multiplication of $S$ on itself extends to an $S$-action on $X$. A toric variety $X$ can be combinatorially described by a fan $\Sigma\subset N_S\otimes_{\Z}\R$, that is, a finite collection of strongly convex rational polyhedral cones such that 
\begin{itemize}
    \item if $\sigma\in\Sigma$, and $\tau$ is a face of $\sigma$, then $\tau\in\Sigma$;
    \item if $\sigma,\sigma'\in\Sigma$, then $\sigma\cap\sigma'$ is a face of $\sigma$ and $\sigma'$.
\end{itemize}
We denote by $\Sigma(k)$ the set of cones of $\Sigma$ of dimension $k$. Elements of $\Sigma(1)$ are called \emph{rays}. We will abuse notation by denoting with $\rho$ both a ray and its minimal generator in $N_S$.

\begin{definition}\label{definition:demazure}
    Let $\Sigma$ be a complete fan.  
    A \emph{Demazure root} of $\Sigma$ is an element $m\in M_S$ for which there exists a ray $\rho\in\Sigma(1)$ such that $\langle m,\rho \rangle=-1$, and such that $\langle m,\rho'\rangle\geq 0$ for any $\rho' \in \Sigma(1)$ with $\rho'\neq \rho$.
    We denote by $\cR(\Sigma)$ the set of Demazure roots of $\Sigma$.
\end{definition}

By definition, given a Demazure root $m\in\cR(\Sigma)$ there exists a unique ray $\rho\in\Sigma(1)$ associated to $m$, which we will denote by $\rho_m$. 

We remark that the original definition of Demazure (cf. \cite[Définition~4, \S~4]{demazure} requires an extra condition, which is automatically satisfied in the complete case (see \cite[Remarques~3, \S~4]{demazure}).

As we will see, any Demazure root induces a $\G_a$-action on a toric variety associated to the fan $\Sigma$. To do so, we first recall the notion of $L$-normalized $\G_a$-action.

\begin{definition}
    Let $L$ be a linear algebraic group, let $\omega\colon L \rightarrow \C^{\times}$ be a character of $L$, and let $Y$ be an $L$-variety.
    A $\G_{a}$-action $A\colon \G_{a} \times Y \rightarrow Y$ is called \emph{$L$-normalized with weight $\omega$} if
    \[
        l. A(s,\, l^{-1}.y) = A(\omega(l) s,\, y), \text{ for every } l\in L,\, s \in \G_{a},\, y \in Y.
    \]
    The $\G_{a}$-subgroup of $\Aut(Y)$ associated to $A$ is called an \emph{$L$-root subgroup}.
\end{definition}

Remark that a $\G_{a}$-action is $L$-normalized if and only if the associated $\G_{a}$-subgroup of the automorphism group is normalized by $L$.

\begin{definition}[{\cite[Definition~3.1]{nill}}]\label{definition:semisimpleUnipotent}
    Define
    \begin{align*}
        \cS(\Sigma) &\coloneqq \{m \in \cR(\Sigma) : -m \in \cR(\Sigma)\}, \\
        \cU(\Sigma) &\coloneqq \cR(\Sigma) \setminus \cS(\Sigma).
    \end{align*}
    Elements of $\cS(\Sigma)$ and $\cU(\Sigma)$ are called \emph{semisimple} and \emph{unipotent}, respectively.
\end{definition}

\begin{theorem}[{\cite[Th\'eor\`eme~4, \S4]{demazure}}]\label{theorem:liendo}
    Let $X$ be a smooth complete $S$-toric variety with fan $\Sigma$. 
    \begin{enumerate}
        \item\label{item:liendo1} For a Demazure root $m\in \cR(\Sigma)$, the rational map 
        $$\G_a\times S\dashrightarrow S, \quad (s,t)\mapsto t\lambda_{\rho_m}(1+s\chi^m(t))$$
        extends to a faithful $S$-normalized $\G_a$-action $A_{m}$ of weight $m$ on $X$. 
        
        \item\label{item:liendo2} If $U_{m}$ is the $S$-root subgroup of $\Aut^{0}(X)$ associated to $A_{m}$, then
        \[
            \Lie (\Aut^{0}(X)) = \Lie(S) \oplus \bigoplus_{m \in \cR(\Sigma)} \Lie(U_{m}).
        \]
        as $S$-modules. In particular, 
        \[ \dim(\Aut^0(X)) = \dim(S) + |\cR(\Sigma)|. \]
        \item\label{item:liendo3} The $S$-root subgroups $U_{m}$ for $m \in \cU(\Sigma)$ generate the unipotent radical $R^{u}(\Aut^{0}(X))$, and $S$ and $U_{m}$ for $m \in \cS(\Sigma)$ generate a Levi subgroup of $\Aut^{0}(X)$.
        
        \item $\Aut^{0}(X)$ is reductive if and only if every Demazure root of $\Sigma$ is semisimple.
    \end{enumerate}
\end{theorem}

\begin{remark}
    Theorem~\ref{theorem:liendo} is obtained by Demazure \cite{demazure}, and it has been generalized to singular cases by various researchers.
    We refer to \cite[pp.~771-772]{nill} and \cite[Introduction]{liendo} for the history.
\end{remark}

\subsection{Horospherical varieties}\label{subsection:horospherical}

We recall the notion of horospherical variety and fix the notation we will use in the rest of the paper. We refer to \cite{pasquier} and \cite{Timashev} for details.

\begin{definition} \label{definition:horospherical}
Let $G$ be a connected reductive linear algebraic group. 
A closed subgroup $H$ of $G$ is called \emph{horospherical} if $H$ contains the unipotent radical of a Borel subgroup of $G$. A normal $G$-variety $X$ is called \emph{horospherical} if $X$ contains an open dense $G$-orbit isomorphic to $G/H$ for a horospherical subgroup $H$ of $G$.
\end{definition}

\begin{proposition}[{\cite[Proposition 2.2]{pasquier}}] \label{prop:normalizer-parabolic}
    Let $H$ be a horospherical subgroup of $G$ containing $R^{u}(B)$ for a Borel subgroup $B$ of $G$.
    The normalizer $N_{G}(H)$ of $H$ in $G$ is a unique parabolic subgroup $P$ of $G$ such that $H$ is the intersection of kernels of some characters $P \rightarrow \C^{\times}$.
    Furthermore, $N_{G}(H)$ contains $B$, and the quotient group $N_{G}(H)/H$ is a torus.
\end{proposition}

From now on, we fix $H$ to be a horospherical subgroup of $G$, and set $P\coloneqq N_{G}(H)$ and $S\coloneqq P/H$.

\begin{definition}
    The character lattice $M_{S}$ of the torus $S$ is called the \emph{weight lattice} of $G/H$, and denoted by $M_{G/H}$.
    Its dual lattice is denoted by $N_{G/H}$.
\end{definition}

In geometric terms, $G/H$ admits a $G$-equivariant morphism $\varphi\colon G/H \rightarrow G/P$ whose fibers are isomorphic to the torus $S$.
That is, $G/H$ can be regarded as an equivariant torus bundle over a rational homogeneous space.
Let us also remark the following elementary property:

\begin{proposition} \label{prop:fiber is torus}
    With respect to the $P$-action on $\varphi^{-1}(e  P) = S$, $H$ acts trivially on $S$.
    Thus the $P$-action induces an $S$-action on $\varphi^{-1}(e  P)$ with a single orbit isomorphic to $S$ itself.
\end{proposition}

We briefly review the Luna--Vust theory (see \cite{LV}) for horospherical varieties.
To do this, observe that by the Bruhat decomposition, for any Borel subgroup $B$ of $G$, $G/H$ contains only finitely many $B$-orbits.
In particular, $G/H$ contains an open dense $B$-orbit, and only finitely many $B$-invariant prime divisors.

\begin{definition}\label{definition:color}
    Let $B$ be a Borel subgroup of $G$.
    Suppose that $X$ is a horospherical $G$-variety with an open dense $G$-orbit $G/H$.
    \begin{enumerate}
        \item $B$-invariant prime divisors on $G/H$ are called \emph{$B$-colors} of $G/H$.
        Equivalently, $B$-colors of $G/H$ are preimages of the $B$-Schubert divisors via the projection $\varphi\colon G/H \rightarrow G/P$.
        The set of $B$-colors on $G/H$ is denoted by $\cD^{B}$.

        \item The closures of $B$-colors of $G/H$ in $X$ are called \emph{$B$-colors} of $X$.

        \item $X$ is called \emph{toroidal} if no $B$-color of $X$ contains a $G$-orbit of $X$. Notice this property does not depend on the choice of $B$.
    \end{enumerate}
\end{definition}

\begin{remark}
    In the literature concerning the so-called \emph{colored fan}, the terminology \emph{color} refers to a color (in our definition) that contains a $G$-orbit, and accordingly, toroidal horospherical varieties are called horospherical varieties \emph{without color}.
    These terminologies are not used in this paper.
\end{remark}

Since $G/H$ contains the open dense $B$-orbit, we have a short exact sequence of groups
\[
    1 \rightarrow \C^{\times} \xrightarrow{\iota} (\C(G/H)^{\times})^{(B)} \xrightarrow{\mu} \Lambda^{B}_{G/H} \rightarrow 1
\]

where
\begin{itemize}   
    \item $(\C(G/H)^{\times})^{(B)}$ is the multiplicative group of nonzero rational $B$-eigenfunctions on $G/H$, that is,
    \[ (\C(G/H)^{\times})^{(B)} = \{ f\in\C(G/H)\setminus \{0\}\colon \exists \chi_f\in M_B \text{ s.t. } b.f=\chi_{f}(b) f, \, \forall b\in B\}; \]
    
    \item $\Lambda^{B}_{G/H}$ is the sublattice of $M_{B}$ defined  by
    \[ \Lambda^{B}_{G/H} = \{\chi_f \in M_B\colon f\in (\C(G/H)^\times)^{(B)}\}; \]
    \item $\mu$ sends $f$ to $\chi_{f}$, and $\iota$ sends $\C^{\times}$ to the subgroup of nonzero constant functions on $G/H$.
\end{itemize}

Using horosphericality of $G/H$, one may identify $\Lambda^{B}_{G/H}$ and $M_{G/H}$ (see \cite[\S7.2]{Timashev}); we briefly recall the idea below.
Choose a Borel subgroup $B^{-}$ such that $H$ contains $R^{u}(B^{-})$, its maximal torus $T$ and the opposite Borel subgroup $B^{+}$.
Note that $P$ contains $B^{-}$, and let us denote by $P^{+}$ the parabolic subgroup opposite to $P$.
In $G/P$, $B^{+}.(eP)$ is the open Schubert cell, and its preimage $\varphi^{-1}(B^{+}.(eP))$ can be identified with $R^{u}(P^{+}) \times S$ via the map $R^{u}(P^{+}) \times S \rightarrow \varphi^{-1}(B^{+}.(eP))$, $(u,\, s) \mapsto u.s$ (cf. \cite[Corollary~3.5]{Kempf}).
In particular, rational $S$-eigenfunctions on $S=\varphi^{-1}(e P)$ extend to rational $B^{+}$-eigenfunctions on $G/H$.
That is, the restriction map $(\C(G/H)^{\times})^{(B^{+})} \rightarrow (\C(S)^{\times})^{(S)}$ has an inverse, and hence we have the following commutative diagram:
\begin{equation} \label{diagram:weight lattice of G/H}
    \begin{tikzcd}
        1 \arrow[r] & \C^{\times} \arrow[r, "\iota"] \arrow[d, "="] & (\C(G/H)^{\times})^{(B^{+})} \arrow[r, "\mu"] \arrow[d, "\simeq"] & \Lambda^{B^{+}}_{G/H} \arrow[r] & 1 \\
        1 \arrow[r] & \C^{\times} \arrow[r] & (\C(S)^{\times})^{(S)} \arrow[r] & \Lambda^{S}_{S} = M_{G/H} \arrow[r] & 1
    \end{tikzcd}
\end{equation}
Therefore we obtain an isomorphism $\Lambda^{B^{+}}_{G/H} \rightarrow M_{G/H}$ defined by $\chi_{f} \mapsto \chi_{(f|_{S})}$.

For any prime divisor $D$ in $X$, we denote by $\nu_D$ its associated valuation. Since $\nu_D(\iota(\C^\times))=0$, it induces a linear functional on $\Lambda^{B^{+}}_{G/H}$, and hence on $M_{G/H}$.

\begin{definition}\label{definition:colormap}
    We define the \emph{color map} 
    \[ \epsilon^{+}\colon \cD^{B^{+}} \rightarrow N_{G/H}\]
    by sending a $B^{+}$-color $D$ to the induced functional $\epsilon^+(D)$ on $M_{G/H}$.
\end{definition}

One can compute the color map using the description of $\cD^{B^{+}}$ as the set of preimages of $B^{+}$-Schubert divisors on $G/P$.
To do this, let $\Pi^{+}$ be the set of simple roots of $G$ with respect to $B^{+}$ (and $T$), and let $\Pi^{+}_{P}$ be its subset consisting of simple roots that are not roots of $P$.
For each $\alpha \in \Pi^{+}_{P}$ and the Weyl element $s_{\alpha}$ corresponding to the simple reflection with respect to $\alpha$, $\overline{B^{+}s_{\alpha} P}$ is a $B^{+}$-Schubert divisor, and hence $D_{\alpha}\coloneqq \varphi^{-1}(\overline{B^{+}s_{\alpha}P}) \in \cD^{+}$.

\begin{proposition}[{\cite[Proposition 20.13]{Timashev}}] \label{prop:color is coroot}
    For $m \in M_{G/H}$, the value of the pairing between $m$ and $\epsilon^{+}(D_{\alpha})$ is equal to the value of pairing between the character $T \twoheadrightarrow S \xrightarrow{\chi^{m}} \C^{\times}$ and the simple coroot $\check{\alpha}$.
\end{proposition}

\subsection{Toroidal horospherical varieties}\label{subsection:ToroidalHorospherical}

From now on, we focus on the case of toroidal horospherical varieties.
The following theorem is a consequence of the Luna--Vust theory:
\begin{theorem}[{\cite[Exemple~2.3]{pasquier}, \cite[Theorem 15.10]{Timashev}}] \label{thm:toroidal horo as fiber bundle}
    Let $X$ be a toroidal horospherical variety with an open dense $G$-orbit $G/H$.
    The natural projection $\varphi\colon G/H \rightarrow G/P$ extends to a $G$-equivariant morphism $\Phi\colon X \rightarrow G/P$.
    As a consequence, $X$ is isomorphic to the parabolic induction from $F$ via $P\to S$, that is the $G$-equivariant fiber bundle over $G/P$ defined as
    \[
        G\times^{P}F\coloneqq (G \times F)/ \sim,
    \]
    where $F\coloneqq \Phi^{-1}(eP)$ and $(g,\,x) \sim (gp^{-1},\, p.x)$ for every $p\in P$.
\end{theorem}

In the following, for a $G$-variety $X$, we denote by $\Aut_{G}(X)$ the group of $G$-equivariant automorphisms of $X$, that is,
\[ \Aut_{G}(X) \coloneqq \{\varphi \in \Aut(X)\colon \varphi(g. x) = g. \varphi(x), \text{ for every } g \in G,\, x \in X\}.\]

\begin{theorem}[{\cite[Proposition 1.8]{Timashev}, \cite[Corollary 5.4]{Knop}}]
\label{thm:equiv aut}
    Let $X$ be a toroidal horospherical variety with an open dense $G$-orbit $G/H$.
    There are isomorphisms $\Aut_{G}(X) \simeq \Aut_{G}(G/H) \simeq S$ where the first isomorphism is the restriction map, and the second is given by $S = P/H \rightarrow \Aut_{G}(G/H)$, $p  H \mapsto (g  H \mapsto (g  p^{-1})  H)$.
\end{theorem}

Let $X$ be a toroidal horospherical variety with an open dense orbit $G/H$, and $\Phi\colon X \rightarrow G/P$ the morphism extending $\varphi\colon G/H \rightarrow G/P$.
Put $F\coloneqq \Phi^{-1}(e  P)$ so that $X \simeq G\times^{P}F$.
The following well-known proposition is an immediate consequence of Proposition \ref{prop:fiber is torus}:
\begin{proposition}\label{proposition:FiberisS-ToricVariety}
    The $P$-action on $F$ factors through $P \twoheadrightarrow S$, and $F$ is a toric variety with respect to the $S$-action.
\end{proposition}

Therefore toroidal horospherical varieties are equivariant toric variety bundles over rational homogeneous spaces.
Conversely, for any $S$-toric variety $F'$, the parabolic induction $G \times^{P} F'$ from $F'$ via $P\to S$ is a toroidal horospherical $G$-variety with an open orbit $G/H$ (cf. \cite[Exemple~2.3]{pasquier}).

To sum up, a toroidal horospherical variety with an open dense orbit $G/H$ comes equipped with the color map $\epsilon^{+}\colon \cD^{B^{+}} \rightarrow N_{G/H}$ and a fiber bundle structure $X \simeq G \times^{P}F \xrightarrow{\Phi} G/P$ for an $S$-toric variety $F$.
The relationship between $\epsilon^{+}(\cD^{B^{+}})$ and the fan of $F$ shall play a significant role in our description of $\Aut^{0}(X)$ (see Section \ref{section:maintheorem}).

\subsection{Automorphism groups}

From now on, we further assume that $X$ is complete so that $\Aut^{0}(X)$ is a linear algebraic group.

Our goal is to study $\Aut^{0}(X)$ using $\Aut^{0}(G/P)$ and $\Aut^{0}(F)$ via the $F$-fiber bundle $\Phi\colon X \rightarrow G/P$.
First, since $\Phi\colon X \rightarrow G/P$ is a morphism with connected fibers, by Blanchard's lemma (see \cite{blanchard}, \cite[Proposition~4.2.1]{BrionBlanchard}), the $\Aut^{0}(X)$-action on $X$ descends to an $\Aut^{0}(X)$-action on $G/P$, i.e., we have a map $\Phi_{*}\colon \Aut^{0}(X) \rightarrow \Aut^{0}(G/P)$.
The structure of $\Aut^{0}(G/P)$ is well understood:
\begin{theorem}[{\cite{DemazureRHS}, \cite[Theorems 3.3.1, 3.3.2]{AkhBook}}] \label{thm:Aut G/P}
    Let $L$ be a connected reductive linear algebraic group, and $Q$ a parabolic subgroup.
    \begin{enumerate}
        \item If $L$ acts on the rational homogeneous space $L/Q$ faithfully, then $L$ is a semisimple Lie group of adjoint type.
        Moreover, for the decomposition $L = \prod_{i}L_{i}$ into the simple factors, if $Q_{i}\coloneq Q \cap L_{i}$, a parabolic subgroup of $L_{i}$, then $L/Q \simeq \prod_{i}L_{i}/Q_{i}$, and thus
        \[
            \Aut^{0}(L/Q) \simeq \prod_{i} \Aut^{0}(L_{i}/Q_{i}).
        \]
        \item If $L$ is simple, then the morphism $L \rightarrow \Aut^{0}(L/Q)$ is surjective, except in the following cases:
        \begin{enumerate}
            \item $L$ is of type $B_{l}$ ($l \ge 3$) and $Q$ is the maximal parabolic associated to the unique short simple root, i.e., its marked Dynkin diagram is

            \begin{center}
            \begin{tikzpicture}[scale=2]
                \dynkin B{*.*x}.
            \end{tikzpicture}
            \end{center}
            
            In this case, $\Aut^{0}(L/Q)$ is the adjoint group of type $D_{l+1}$.
            \item $L$ is of type $C_{l}$ ($l \ge 2$) and $Q$ is the maximal parabolic associated to the unique short simple root that represents an end node in the Dynkin diagram of $L$, i.e., its marked Dynkin diagram is

            \begin{center}
            \begin{tikzpicture}[scale=2]
                \dynkin C{x*.**}.
            \end{tikzpicture}
            \end{center}
            
            In this case, $\Aut^{0}(L/Q)$ is the adjoint group of type $A_{2l-1}$.
            \item $L$ is of type $G_{2}$ and $Q$ is the maximal parabolic associated to the unique short simple root, i.e., its marked Dynkin diagram is

            \begin{center}
            \begin{tikzpicture}[scale=2]
               \dynkin G{x*}.
            \end{tikzpicture}          
            \end{center}
            
            In this case, $\Aut^{0}(L/Q)$ is the adjoint group of type $B_{3}$.
        \end{enumerate}
    \end{enumerate}
\end{theorem}

Since $\Aut^{0}(G/P)$ is equal to the image of $G$ in $\Aut(G/P)$ in most cases, one may expect that $\Phi_{*}$ is surjective.
In fact, due to Brion's result \cite{brionWonderful}, we can say more.
To state the result, let $\partial_{G}X$ be the complement of the open orbit $G/H$ in $X$, which is indeed a divisor (see \cite[Theorem 29.1]{Timashev}).
Define $\Aut(X, \, \partial_{G}X)$ to be the subgroup of $\Aut(X)$ consisting of automorphisms stabilizing every component of $\partial_{G}X$, and let $\Aut^{0}(X, \, \partial_{G}X)$ be its identity component.
The connected group $\Aut^{0}(X, \, \partial_{G}X)$ contains both the image of $G \rightarrow \Aut^{0}(X)$ and $\Aut_{G}(X)$, and stabilizes all $G$-orbits in $X$.

\begin{theorem}[{\cite[Theorem 4.4.1]{brionWonderful}}] \label{thm:brion}
    $\Phi_{*}$ sends $\Aut^{0}(X, \, \partial_{G}X)$ onto $\Aut^{0}(G/P)$, and the identity component of its kernel is $\Aut_{G}(X)$.
    Moreover, $\Aut_{G}(X)$ is the identity component of the center of $\Aut^{0}(X, \, \partial_{G}X)$.
    In particular, $\Aut^{0}(X, \, \partial_{G}X)$ is reductive.
\end{theorem}

\begin{corollary}\label{eqn:Blanchard}
    We have a short exact sequence:
    \begin{equation*} 
        \begin{tikzcd}
        1 \arrow[r] & K \coloneq \ker (\Phi_{*}) \arrow[r] &\Aut^{0}(X) \arrow[r, "\Phi_{*}"] & \Aut^{0}(G/P) \arrow[r] & 1.
    \end{tikzcd}
\end{equation*}
\end{corollary}

Since $\Aut^{0}(G/P)$ is well understood (see Theorem \ref{thm:Aut G/P}), our main interest lies in $K$.
Geometrically, $K$ consists of fiberwise automorphisms belonging to $\Aut^{0}(X)$, and so we have a natural restriction map
\[
    \res\colon K \rightarrow \Aut(F), \quad f \mapsto f|_{F}.
\]

The following proposition seems well known to experts (see \cite[Proofs of Theorem 9.1 and Proposition 10.1]{PezziniReductive}).
For the reader's convenience, let us record its proof.

\begin{proposition} \label{prop-unip-kernel}
	We have the following.
	\begin{enumerate}
		\item \label{item:prop-unip-kernel 1} $\Aut_{G}(X)$ is a maximal torus of $K^{0}$.

		\item \label{item:prop-unip-kernel 2} The kernel of $\res \colon K^{0} \rightarrow \Aut^{0}(F)$ is unipotent.
	\end{enumerate}
\end{proposition}

\begin{proof}
	\begin{enumerate}
		\item Recall that $\Aut_{G}(X)$ is a torus isomorphic to $S=P/H$ by Theorem \ref{thm:equiv aut}.
        By the Borel fixed point theorem, $\Aut_{G}(X)$ fixes a point in $G/P$, which implies that $\Aut_{G}(X)$ acts trivially on $G/P$, since by definition its action commutes with the $G$-action.
		That is, $\Aut_{G}(X)$ is a torus contained in $K^{0}$.

		Let $S'$ be a maximal torus of $K^{0}$, containing $\Aut_{G}(X)$.
		We assert that $S'$ coincides with $\Aut_{G}(X)$.
		Since the $K^{0}$-action preserves each fiber of $\Phi$, for any $x \in G/P$ and $F_{x} \coloneqq \Phi^{-1}(x)$, we have the restriction map
		\[
			\res_{x} \colon S' \rightarrow \Aut^{0}(F_{x}),
		\]
		and its image is a torus containing the image of $\Aut_{G}(X)$.
		In fact, since $F_{x}$ is a complete toric variety with respect to the action of $\Aut_{G}(X)$ (cf. Proposition \ref{proposition:FiberisS-ToricVariety}), the image of $\Aut_{G}(X)$ in $\Aut^{0}(F_{x})$ is a maximal torus, and hence
		\[
			\res_{x}(S') = \res_{x}(\Aut_{G}(X)).
		\]
		That is, the $S'$-action on $F_{x}$ factors through the $\Aut_{G}(X)$-action on $F_{x}$.
		Therefore $S'$ is contained in $\Aut^{0}(X,\, \partial_{G}X)$, since if $D$ is a component of $\partial_{G}X$, then
		\[
			S'.D = \bigcup_{x \in G/P} S'. (D \cap F_{x}) = \bigcup_{x \in G/P} \Aut_{G}(X). (D \cap F_{x}) = D.
		\]
		However, since the connected component of $K^{0} \cap \Aut^{0}(X,\, \partial_{G}X)$ is $\Aut_{G}(X)$ by Theorem \ref{thm:brion}, we conclude that $S' = \Aut_{G}(X)$, i.e., $\Aut_{G}(X)$ is a maximal torus of $K^{0}$.

		\item It suffices to show that the kernel of $\res \colon K^{0} \rightarrow \Aut^{0}(F)$ does not contain a torus of positive dimension as an algebraic subgroup.
		Let $S''$ be a torus contained in the kernel of $\res$.
		By (\ref{item:prop-unip-kernel 1}), there exists $g \in K^{0}$ such that $g S''g^{-1} \subset \Aut_{G}(X)$.
		Note that $gS''g^{-1}$ is still in the kernel of $\res$.
		However, since $F$ is an $\Aut_{G}(X)$-toric variety, the $\Aut_{G}(X)$-action on $F$ is faithful, and thus $gS''g^{-1} = \{id\}$.
		That is, $S'' = \{id\}$, and the statement follows.
	\end{enumerate}
\end{proof}

\section{Main results}\label{section:maintheorem}

Throughout this section, we use the notation in Setup \ref{setup}.

\begin{setup} \label{setup}
    Let $G$ be a connected reductive linear algebraic group, $T$ a maximal torus of $G$, and $B^{\pm}$ Borel subgroups of $G$ containing $T$ that are opposite to each other.
    Let $H$ be a horospherical subgroup of $G$ containing $R^{u}(B^{-})$.
    Define the normalizer $P \coloneqq N_{G}(H)$ and the torus $S \coloneqq P/H$.
    
    Define $X$ to be a smooth complete toroidal horospherical $G$-variety with open $G$-orbit $G/H$, and let
    $\Phi : X \rightarrow G/P$ be the $G$-equivariant morphism extending $G/H \rightarrow G/P$, and $F \coloneqq \Phi^{-1}(eP)$.
    We denote by $\cR_{S}(F)$ the set of Demazure roots of $F$ as an $S$-toric variety.
\end{setup}

The set of $B^{+}$-colors on $G/H$ is denoted by $\cD^{+} = \cD^{B^{+}}$, and the color map with respect to $B^{+}$ is denoted by $\epsilon^{+}\colon \cD^{+} \rightarrow N_{G/H}$.
In the following, we identify $M_{G/H}$ with the image of injection $M_{G/H} \hookrightarrow M_{T}$ induced by the natural quotient $T (\subset P) \twoheadrightarrow S$.
We also identify $M_{T}$ and $M_{B^{+}}$.

First, in analogy with Definition \ref{definition:semisimpleUnipotent}, we introduce a notion of semisimple and unipotent roots for toroidal horospherical varieties:

\begin{definition}\label{definition:horospherical roots}
    Define
    \begin{align*}
    \cR^{+}_{G}(X) & \coloneqq \{m \in \cR_{S}(F) : \langle m,\, \epsilon^{+}(\cD^{+}) \rangle \ge 0\},\\
    \cS^{+}_{G}(X) & \coloneqq \{m \in \cR^{+}_{G}(X) \colon - m \in \cR^{+}_{G}(X) \}, \\
    \cU^{+}_{G}(X) &\coloneqq \cR^{+}_{G}(X) \setminus \cS^{+}_{G}(X).
    \end{align*}
    Elements of $\cR^{+}_{G}(X)$ are called \emph{$B^{+}$-roots} of $X$.
    Elements of $\cS^{+}_{G}(X)$ and $\cU^{+}_{G}(X)$ are said to be \emph{semisimple} and \emph{unipotent}, respectively.
\end{definition}

We provide a criterion for $S$-normalized $\G_{a}$-actions on $F$ to extend to $B^{+}$-normalized $\G_{a}$-actions on $X$ in terms of $\cR^{+}_{G}(X)$:

\begin{theorem} \label{theorem:B-normalized Ga action}
    For $m \in \cR_{S}(F)$, the following are equivalent:
    \begin{enumerate}
        \item\label{item:theorem:B-normalized Ga action} There exists a nontrivial $B^{+}$-normalized $\G_{a}$-action $A^{+}_{m}$ on $X$ with weight $m$, preserving every fiber of $\Phi$.
        \item $m \in \cR^{+}_{G}(X)$.
    \end{enumerate}
    If one of the two conditions holds, then the $\G_{a}$-action $A_{m}^{+}$ satisfying (\ref{item:theorem:B-normalized Ga action}) is unique, and $A^{+}_{m}|_{F}$ is a nontrivial $S$-normalized $\G_{a}$-action on $F$ with weight $m$.
    Furthermore, $A^{+}_{m}$ is $G$-normalized if and only if $\langle m,\, \epsilon^{+}(\cD^{+}) \rangle = 0$.
\end{theorem}

In particular, since $\langle \cS^{+}_{G}(X),\, \epsilon^{+}(\cD^{+}) \rangle = 0$, the $\G_{a}$-action $A^{+}_{m}$ for a semisimple $B^{+}$-root $m$ is $G$-normalized.

Indeed, Theorem \ref{theorem:B-normalized Ga action} is a consequence of the following lemma:

\begin{lemma} \label{lemma: key lemma}
    For a $B^{+}$-dominant weight $\omega$ of $T$, let $V(\omega)$ be the irreducible $G$-module with $B^{+}$-highest weight $\omega$.
    \begin{enumerate}
        \item\label{item:keylemma1} For $m \in M_{G/H}$, $m$ is $B^{+}$-dominant if and only if $\langle m, \, \epsilon^{+}(\cD^{+}) \rangle \ge 0$.
        In this case, $\langle m, \, \epsilon^{+}(\cD^{+}) \rangle = 0$ if and only if $\dim V(m) = 1$.

        \item\label{item:keylemma2} As $G$-modules,
        \[
            \Lie(K^{0}) = \Lie (\Aut_{G}(X)) \oplus \bigoplus_{m \in \cR^{+}_{G}(X)} V(m).
        \]
        \item\label{item:keylemma3} For $m \in \cR^{+}_{G}(X)$, the $B^{+}$-highest weight line of $V(m)$ is the Lie algebra of a $B^{+}$-root subgroup $U^{+}_{m}$ of $K^{0}$ whose associated $\G_{a}$-action on $X$ is $B^{+}$-normalized with weight $m$.
        Furthermore, the restriction of the action of $U^{+}_{m}$ on $F$ is nontrivial and $S$-normalized with weight $m$.
    \end{enumerate}
\end{lemma}

\begin{proof}
	\begin{enumerate}
		\item By definition, $m$ is $B^{+}$-dominant if and only if $\langle m,\, \check{\alpha} \rangle \ge 0$ for any simple root $\alpha \in \Pi^{+}$.
        Since $m \in M_{G/H}$, which is the character lattice of $P/H$, the pairing between $m$ and a simple coroot of $P$ is zero.
        That is, $m$ is $B^{+}$-dominant if and only if $\langle m,\, \check{\alpha} \rangle \ge 0$ for any $\alpha \in \Pi^{+}_{P}$.
        By Proposition \ref{prop:color is coroot}, it is equivalent to $\langle m,\, \epsilon^{+}(\cD^{+}) \rangle \ge 0$.
        For the second statement, observe that $\langle m,\, \epsilon^{+}(\cD^{+}) \rangle = 0$ if and only if $\langle m,\, \check{\alpha} \rangle = 0$ for any $\alpha \in \Pi^{+}$, which is equivalent to $\dim V(m) = 1$.

		\item Recall that there are natural Lie algebra isomorphisms $H^{0}(X,\, T_{X}) \simeq \Lie (\Aut^{0}(X))$ and $H^{0}(X,\, \Phi^{*}T_{G/P}) \simeq H^{0}(G/P,\, T_{G/P}) \simeq \Lie (\Aut^{0}(G/P))$.
		Moreover, for the relative tangent sequence
		\[
		0 \rightarrow T_{\Phi} \rightarrow T_{X} \xrightarrow{d \Phi} \Phi^{*} T_{G/P} \rightarrow 0,
		\]
		the Lie algebra isomorphisms fit into the commutative diagram
		\begin{center}
		\begin{tikzcd}
		\Lie(\Aut^{0}(X)) \arrow[r, "d \Phi_{*}"] \arrow[d, "\simeq"] & \Lie(\Aut^{0}(G/P)) \arrow[d, "\simeq"] \\
		H^{0}(X,\, T_{X}) \arrow[r, "d \Phi"] & H^{0}(X,\, \Phi^{*}T_{G/P}),
		\end{tikzcd}
		\end{center}
		and hence by Corollary \ref{eqn:Blanchard} we have a natural isomorphism
		\[
		\Lie (K^{0}) \simeq H^{0}(X,\, T_{\Phi}),
		\]
		which is $G$-equivariant.
        Here, $G$ acts on $K^{0}$ via the conjugation, and the $G$-action on $H^{0}(X,\, T_{\Phi})$ is given by taking differentials of the $G$-action on $X$.

        To compute $H^{0}(X,\, T_{\Phi})$ as a $G$-module, we apply the Borel--Weil--Bott theorem, using the natural isomorphism $H^{0}(X,\, T_{\Phi}) \simeq H^{0}(G/P,\, \Phi_{*} T_{\Phi})$.
		In fact, by the Grauert theorem, $\Phi_{*}T_{\Phi}$ is a $G$-equivariant vector bundle on $G/P$ whose fiber at $e P$ is, thanks to Theorem \ref{theorem:liendo},
		\begin{align*}
		H^{0}(F,\, T_{\Phi}|_{F}) &\simeq H^{0}(F,\, T_{F}) \\
        &\simeq \Lie(S) \oplus \bigoplus_{m \in \cR_{S}(F)} \C_{m}     
		\end{align*}
		as an $S$-module where $\C_{m}$ means the 1-dimensional $S$-module with weight $m$.
		Therefore by the Borel--Weil--Bott theorem and by (\ref{item:keylemma1}), we have $G$-module isomorphisms
		\begin{align*}
		\Lie(K^{0}) &\simeq H^{0}(G/P,\, \Phi_{*} T_{\Phi}) \\
		&\simeq V(0)^{\oplus \dim S} \oplus \bigoplus_{\substack{m \in \cR_{S}(F) \\ \text{$m$ is $B^{+}$-dominant}}} V(m) \\
		& = V(0)^{\oplus \dim S} \oplus \bigoplus_{\substack{m \in \cR_{S}(F) \\ \langle m,\, \epsilon^{+}(\cD^{+}) \rangle \ge 0}} V(m) \\
        & = V(0)^{\oplus \dim S} \oplus \bigoplus_{m \in \cR^{+}_{G}(X)} V(m).
		\end{align*}
		Finally, the trivial factor $V(0)^{\oplus \dim S}$ coincides with $\Lie (\Aut_{G}(X))$, since $\Aut_{G}(X)$ commutes with $G$ and $\Aut_{G}(X) \simeq S$ (see Theorem \ref{thm:equiv aut}).

		\item Let $m \in \cR^{+}_{G}(X)$, i.e., $m$ is a nonzero $B^{+}$-highest weight of $\Lie(K^{0})$.
		By (\ref{item:keylemma2}), there is a unique $B^{+}$-invariant 1-dimensional subspace of $\Lie(K^{0})$ with weight $m$, say $l_{m} \subset \Lie(K^{0})$.
		By the uniqueness of the Jordan decomposition, $l_{m}$ is the Lie algebra of a 1-dimensional algebraic subgroup $U^{+}_{m}$ of $K^{0}$, which is either $\C^{\times}$ or $\G_{a}$.
		Furthermore, since $l_{m}$ is $B^{+}$-invariant, $U^{+}_{m}$ is normalized by $B^{+}$.
		It follows that the action of $U^{+}_{m}$ on $F$ is not trivial, otherwise it acts trivially on $B^{+}.F$, which is open in $X$, a contradiction.
		Moreover, since the action of $U^{+}_{m}$ on $F$ is normalized by $S$ with weight $m$, $U^{+}_{m}$ must be $\G_{a}$ by Theorem \ref{theorem:liendo}.
	\end{enumerate}
\end{proof}

\begin{corollary} \label{coro:LieAut(X)}
    The following hold:
    \begin{enumerate}
        \item\label{item:coro of Key Lemma 1} As $G$-modules,
        \[
        \Lie(\Aut^{0}(X)) = \Lie (\Aut^{0}(G/P)) \oplus \Lie (\Aut_{G}(X)) \oplus \bigoplus_{m \in \cR^+_G(X)} V(m).
        \]
        In particular,  
        \[
	       \dim \Aut^{0}(X) =\dim \Aut^{0}(G/P) + \dim S + |\cS^{+}_{G}(X)| + \sum_{m \in \cU^{+}_{G}(X)} \dim V(m).
        \]

        \item\label{item:coro of Key Lemma 2} $\Aut^{0}(X)$ is generated by $\Aut^{0}(X,\, \partial_{G}X)$ and $U^{+}_{m}$, for $m\in \cR^{+}_{G}(X)$.
    \end{enumerate}
\end{corollary}
\begin{proof}
    (\ref{item:coro of Key Lemma 1}) is a direct consequence of Lemma \ref{lemma: key lemma}.
    (\ref{item:coro of Key Lemma 2}) also follows from Lemma \ref{lemma: key lemma} and Theorem \ref{thm:brion}.
\end{proof}

\subsection{Structure theorem} \label{ssection:structure}

In this section we give a more precise structure theorem for $\Aut^{0}(X)$.

\begin{theorem} \label{thm:structure of K} The following hold:
\begin{enumerate}
    \item\label{Kstructure1} $R^{u}(\Aut^{0}(X)) = R^{u}(K^{0})$, and its Lie algebra is $\bigoplus_{m \in \cU^{+}_{G}(X)} V(m)$ as a $G$-module.
    In particular, $R^{u}(\Aut^{0}(X))$ is generated by $U^{+}_{m}$, for $m \in \cU^{+}_{G}(X)$, and their $G$-conjugates.

    \item\label{Kstructure2} $K^{0}$ contains a Levi subgroup $K^{\text{Levi}}$ containing $\Aut_{G}(X)$ such that
    \[
        \Lie(K^{\text{Levi}}) = \Lie (\Aut_{G}(X)) \oplus \bigoplus_{m \in \cS^{+}_{G}(X)} V(m)
    \]
    as $G$-modules.
    In particular, $K^{\text{Levi}}$ is generated by $\Aut_{G}(X)$ and $U^{+}_{m}$, for $m \in \cS^{+}_{G}(X)$.
\end{enumerate}
\end{theorem}

\begin{proof}
	First of all, the equality $R^{u}(\Aut^{0}(X)) = R^{u}(K^{0})$ is the consequence of Corollary \ref{eqn:Blanchard} and the reductivity of $\Aut^{0}(G/P)$.
	
	To complete the proof, we show the remaining statements (\ref{Kstructure1}) and (\ref{Kstructure2}) at the same time.
	Consider the restriction map $\res\colon K^{0} \rightarrow \Aut^{0}(F)$.
	By Lemma \ref{lemma: key lemma}, the image $\res(K^{0})$ is generated by $S$ and the $S$-root subgroups of $\Aut^{0}(F)$ associated to $m \in \cR^{+}_{G}(X)$.
	In particular, $S$ and the $S$-root subgroups associated to $m \in \cS^{+}_{G}(X)$ generate a Levi subgroup $L$ of $\res(K^{0})$, while $R^{u}(\res(K^{0}))$ is generated by the $S$-root subgroups associated to $m \in \cU^{+}_{G}(X)$ (cf. Theorem~\ref{theorem:liendo}(\ref{item:liendo3})).

	As the first step, we claim that $K^{\text{Levi}}$ in (\ref{Kstructure2}) is well defined.
	More precisely, we show that in $\Lie(K^{0})$, the subspace $\Lie(\Aut_{G}(X)) \oplus \bigoplus_{m \in \cS^{+}_{G}(X)} V(m)$ is a reductive subalgebra, and define $K^{\text{Levi}}$ as the associated connected subgroup.
	To this end, since $\dim V(m) = 1$ for $m \in \cS^{+}_{m}$, it is straightforward to see that the direct sum is a Lie subalgebra of $\Lie(K^{0})$.
	Moreover, it is an algebraic subalgebra, since $\Aut_{G}(X)$ and $U^{+}_{m}$ are algebraic subgroups of $K^{0}$ (see Lemma \ref{lemma: key lemma}).
	The reductivity follows from the observation that the Lie algebra morphism
	\[
		d(\res)\colon \Lie(K^{0}) \rightarrow \Lie(\Aut^{0}(F))
	\]
	induces an isomorphism
	\[
		\Lie(\Aut_{G}(X)) \oplus \bigoplus_{m \in \cS^{+}_{G}(X)} V(m) \rightarrow \Lie(L).
	\]
	Indeed, the morphism is surjective and its kernel is contained in $\Lie(\Aut_{G}(X))$.
	Since $\Lie(\Aut_{G}(X))$ is isomorphically sent to $\Lie(S)$ by Theorem \ref{thm:equiv aut}, it follows that the morphism is an isomorphism.
	Therefore $K^{\text{Levi}}$, the associated subgroup to $\Lie(\Aut_{G}(X)) \oplus \bigoplus_{m \in \cS^{+}_{G}(X)} V(m)$, is a connected reductive subgroup of $K^{0}$.

	Next, we prove that $K^{\text{Levi}}$ is a Levi subgroup, and compute $\Lie(R^{u}(K^{0}))$.
	Consider the surjective morphism
	\[
		K^{0} \xrightarrow{\res} \res(K^{0}) \twoheadrightarrow \res(K^{0}) / R^{u}(\res(K^{0})) \simeq L
	\]
	and denote by $R$ its kernel.
	Since $U^{+}_{m}$ for $m \in \cU^{+}_{G}(X)$ is contained in $R$, we have
	\[
		\Lie(R) = \bigoplus_{m \in \cU^{+}_{G}(X)}V(m).
	\]
	Moreover, the identity component $R^{0}$ is a unipotent subgroup.
	Indeed, since $\Lie(R^{u}(\res(K^{0}))) = \bigoplus_{m \in \cU^{+}_{G}(X)} \C_{m}$, $\res \colon K^{0} \rightarrow \res(K^{0})$ induces a surjection $R^{0} \rightarrow R^{u}(\res(K^{0}))$ whose kernel is unipotent by Proposition \ref{prop-unip-kernel}.
	Now since $R^{0}$ is a connected normal unipotent subgroup of $K^{0}$ such that $K^{0}/R^{0}$ is reductive, we have $R^{0} = R^{u}(K^{0})$, proving (\ref{Kstructure1}).
	Furthermore, since the surjection $K^{0} \rightarrow L$ induces an isomorphism $\Lie(K^{\text{Levi}}) \rightarrow \Lie(L)$, it follows that $K^{\text{Levi}}$ is indeed a Levi subgroup, proving (\ref{Kstructure2}).
\end{proof}

\begin{corollary} \label{coro:reductive}
The following hold:
\begin{enumerate}
        \item\label{item:coro:reductive1}
        $\Aut^{0}(X)$ is reductive if and only if $\cU^{+}_{G}(X) = \emptyset$, i.e. every $B^{+}$-root of $X$ is semisimple.

        \item\label{item:coro:reductive2} The algebraic subgroup of $\Aut^0(X)$ generated by $G$, $\Aut_{G}(X)$ and $U^{+}_{m}$ for $m \in \cS^{+}_{G}(X)$ is connected and reductive, and it is a Levi subgroup of $\Aut^{0}(X)$ if the natural morphism $G \rightarrow \Aut^{0}(G/P)$ is surjective (see Remark \ref{rmk:G->Aut often surjects}).
    \end{enumerate}
\end{corollary}

\begin{proof}
    (\ref{item:coro:reductive1}) follows from Theorem \ref{thm:structure of K}(\ref{Kstructure1}). 
    For (\ref{item:coro:reductive2}), let $K^{\text{Levi}}$ be the Levi subgroup of $K^{0}$ obtained in Theorem \ref{thm:structure of K}.
    Since $K^{\text{Levi}}$ is normalized by $G$, $G$ and $K^{\text{Levi}}$ generate a connected closed subgroup $L$ of $\Aut^{0}(X)$ such that $\Lie(L)$ is the (vector space) sum of the image of $\Lie(G) \rightarrow \Lie(\Aut^{0}(X))$ and
    \[
	\Lie(K^{\text{Levi}}) = \Lie(\Aut_{G}(X)) \oplus \bigoplus_{m \in \cS^{+}_{G}(X)} V(m).
    \]
    Moreover, since we have a short exact sequence of Lie algebras
    \[
	0 \rightarrow \Lie(K^{\text{Levi}}) \rightarrow \Lie(L) \xrightarrow{d\Phi_{*}|_{L}} \Lie (\text{image of }G \rightarrow \Aut^{0}(G/P)) \rightarrow 0,
    \]
    we see that $L$ is reductive.
    Finally, if $G \rightarrow \Aut^{0}(G/P)$ is surjective, then $L$ and $R^{u}(\Aut^{0}(X))$ generate $\Aut^{0}(X)$ by Theorem \ref{thm:structure of K} and Corollary \ref{eqn:Blanchard}.
This implies that $L$ is a Levi subgroup.
\end{proof}

\begin{remark} \label{rmk:G->Aut often surjects}
    Surjectivity of $G \rightarrow \Aut^{0}(G/P)$ is not a restrictive condition, due to the following observations:
    \begin{enumerate}
        \item $G \rightarrow \Aut^{0}(G/P)$ is surjective unless $G/P$ contains one of the exceptions in Theorem \ref{thm:Aut G/P} as an irreducible factor.
        For example, if the Dynkin diagram of the semi-simple part of $G$ is simply laced, then $G \rightarrow \Aut^{0}(G/P)$ is surjective.
        
        \item Even in the case where $G/P$ contains some exceptional irreducible factor, we may assume that $G \rightarrow \Aut^{0}(G/P)$ is surjective, by replacing $G$ by $\Aut^{0}(X,\, \partial_{G}X)$, as shown in the following proposition.
    \end{enumerate}
\end{remark}

\begin{proposition}[See also {\cite[Theorem~4.4.1]{brionWonderful}, \cite[\S4]{PezziniReductive}}] \label{prop:rep-by-Aut-b}
Define $\mathbf{G} \coloneqq \Aut^{0}(X,\, \partial_{G} X)$.
Then $X$ is a toroidal horospherical variety with respect to the $\mathbf{G}$-action, and its associated rational homogeneous space is $G/P$.
More precisely, for $\mathbf{H} \coloneqq \Stab_{\mathbf{G}}(eH)$ and $\mathbf{P} \coloneqq N_{\mathbf{G}}(\mathbf{H})$, $\mathbf{G}/\mathbf{P}$ is homogeneous under the natural $G$-action, and $\Stab_{G}(e\mathbf{P})=P$.
\end{proposition}

\begin{proof}
To ease the notation, denote by $O$ the open $G$-orbit $G/H$ in $X$, i.e., $O = G.(eH)$.
By definition, $\mathbf{G}$ preserves all $G$-orbits in $X$, and thus $O = \mathbf{G}.(eH) \simeq \mathbf{G}/\mathbf{H}$.

First we show that $X$ is $\mathbf{G}$-horospherical.
Then it would imply that $X$ is a toroidal horospherical $\mathbf{G}$-variety by \cite[Theorem~4.4.1(2)]{brionWonderful}.
To show the $\mathbf{G}$-horosphericality, we need to show that for the unipotent radical $\mathbf{U}$ of a Borel subgroup of $\mathbf{G}$, $\mathbf{U}$ has a fixed point in the open $\mathbf{G}$-orbit $O$.
Recall that for the $G$-equivariant morphism $\Phi : X \rightarrow G/P$, the $\mathbf{G}$-action on $X$ descends to a $\mathbf{G}$-action on $G/P$ by the Blanchard lemma, and thus $\Phi$ can be regarded as a $\mathbf{G}$-equivariant morphism.
In particular, its restriction $\varphi : O \rightarrow G/P$ is also $\mathbf{G}$-equivariant.
By the Borel fixed point theorem, $\mathbf{U}$ has a fixed point $p \in G/P$, and so acts on $\varphi^{-1}(p)$.
However, since $\varphi^{-1}(p)$ is isomorphic to the torus $S$, and since any algebraic $\G_{a}$-action on a torus is trivial, we see that $\mathbf{U}$ acts trivially on $\varphi^{-1}(p)$.
In particular, $\mathbf{U}$ has a fixed point in $\varphi^{-1}(p) \subset O$, proving that $O$ is a horospherical $\mathbf{G}$-variety.

Next, we show that $\mathbf{G}/\mathbf{P}$ is $G$-homogeneous.
Denote by $\psi: O \simeq\mathbf{G}/\mathbf{H} \rightarrow \mathbf{G}/\mathbf{P}$ the natural projection.
Regarding $\mathbf{G}/\mathbf{P}$ as a $G$-variety via the natural morphism $G \rightarrow \mathbf{G}$, $\psi$ is also $G$-equivariant.
Since $O = \mathbf{G}.(eH) = G.(eH)$ in $X$, we have
\[
	\mathbf{G}/\mathbf{P} = \psi(\mathbf{G}.(eH)) = \psi(G.(eH)) = G.\psi(eH).
\]

It remains to show that the parabolic subgroup $P' \coloneqq \Stab_{G}(e\mathbf{P})$, that is, the preimage of $\mathbf{P}$ under the map $\alpha: G\rightarrow \mathbf{G}$, coincides with $P$.
Recall that there are characters $\chi_{i} : \mathbf{P} \rightarrow \C^{\times}$ for $i \in I$ such that $\mathbf{H} = \bigcap_{i \in I} \ker \chi_{i}$ (cf. Proposition~\ref{prop:normalizer-parabolic}).
Since $H = \alpha^{-1}(\mathbf{H}) \subset P'$, we see that $H$ is the intersection of the kernels of the characters
\[
	P' \xrightarrow{\alpha} \mathbf{P} \xrightarrow{\chi_{i}} \C^{\times}.
\]
By Proposition~\ref{prop:normalizer-parabolic}, we conclude that $P' = N_{G}(H) = P$.
\end{proof}

In \cite[Theorem~4.4.1]{brionWonderful}, it is moreover shown that the colors of $X$ as a $\mathbf{G}$-variety are the same as those of $X$ as a $G$-variety.
This yields the following description of a Levi subgroup of $\Aut^{0}(X)$ without any additional condition.

\begin{corollary}\label{corollary:LeviSubgroup}
	A Levi subgroup of $\Aut^{0}(X)$ is generated by $\Aut^{0}(X,\, \partial_{G}X)$ and $U^{+}_{m}$ for $m \in \cS^{+}_{G}(X)$.
\end{corollary}

\begin{proof}
	We keep the notation of Proposition~\ref{prop:rep-by-Aut-b}.
    Let $O$ be the open $G$-orbit in $X$, and denote by $\alpha: G \rightarrow \mathbf{G}$ the natural morphism.
	Let $\mathbf{B}^{+}$ be a Borel subgroup of $\mathbf{G}$ containing $\alpha(B^{+})$.
    Since $(\C(O)^{\times})^{(\mathbf{B}^{+})}$ is a subgroup of $(\C(O)^{\times})^{(B^{+})}$ (cf. Section~\ref{subsection:horospherical}), $\Lambda_{O}^{\mathbf{B}^{+}}$ can be identified with a sublattice of $\Lambda_{O}^{B^{+}}$ via the embedding $\chi_{f} \mapsto \chi_{f} \circ \alpha$.
	By \cite[Theorem~4.4.1]{brionWonderful}, $\cD^{\mathbf{B}^{+}} = \cD^{B^{+}}$, i.e., $\mathbf{B}^{+}$-colors are same with $B^{+}$-colors, or equivalently, the open $B^{+}$-orbit $B^{+}.(e\mathbf{H}) \subset O$ is also a $\mathbf{B}^{+}$-orbit.
    Hence $\epsilon^{\mathbf{B}^{+}}(D) = \epsilon^{B^{+}}(D)|_{\Lambda_{O}^{\mathbf{B}^{+}}}$ for any $D \in \cD^{\mathbf{B}^{+}} = \cD^{B^{+}}$.

	Now using Proposition~\ref{prop:rep-by-Aut-b}, we consider $\cR^{+}_{\mathbf{G}}(X)$ with respect to $\mathbf{B}^{+}$.
	More rigorously, first we choose another base point $g\mathbf{H} \in O$ such that the unipotent radical of the opposite of $\mathbf{B}^{+}$ (with respect to a maximal torus) fixes $g\mathbf{H}$.
    In other words, the unipotent radical is contained in $\mathbf{H}' \coloneqq \Stab_{\mathbf{G}}(g\mathbf{H})= g\mathbf{H}g^{-1}$.
    In fact, we may assume that $g \in \mathbf{B}^{+}$.
    Indeed, observe that since $\mathbf{B}^{+}.(g\mathbf{H})$ is an open $\mathbf{B}^{+}$-orbit in $O$, it is same with $\mathbf{B}^{+}.(e\mathbf{H})$, and so there exists $b \in \mathbf{B}^{+}$ satisfying $g\mathbf{H} = b\mathbf{H}$ in $\mathbf{G}/\mathbf{H}$, and hence we may replace $g$ by $b$.
	Then for $\mathbf{P}' \coloneqq N_{\mathbf{G}}(\mathbf{H}') = g \mathbf{P}g^{-1}$ and the fiber bundle $\Phi' : X \rightarrow \mathbf{G}/\mathbf{P}'$, $\cR^{+}_{\mathbf{G}}(X)$ is defined as the set of Demazure roots $m' \in M_{\mathbf{P}'/\mathbf{H}'}$ of the $\mathbf{P}'/\mathbf{H}'$-toric variety $F' \coloneqq (\Phi')^{-1}(e\mathbf{P}')$ such that $\langle m',\, \epsilon^{\mathbf{B}^{+}}(\cD^{\mathbf{B}^{+}}) \rangle \ge 0$.
    Notice that $\Phi'$ is precisely the original fiber bundle $\Phi : X \rightarrow \mathbf{G}/\mathbf{P}$ with the different presentation induced by choosing a base point $g\mathbf{P} = \Phi(g\mathbf{H})\in \mathbf{G}/\mathbf{P}$.
    More formally, the fiber bundles $\Phi : X \rightarrow \mathbf{G}/\mathbf{P}$ and $\Phi' : X \rightarrow \mathbf{G}/\mathbf{P}'$ can be identified via the equivariant isomorphism $\mathbf{G}/\mathbf{P} \rightarrow \mathbf{G}/\mathbf{P}'$, $h\mathbf{P}\mapsto hg^{-1}\mathbf{P}'$.
    Namely, $F' = g.F$, and the $\mathbf{B}^{+}$-root subgroups associated to $\cR^{+}_{\mathbf{G}}(X)$ also preserve the fibers of $\Phi$.
    
	Recall that for $m' \in M_{\mathbf{P}'/\mathbf{H}'}$, the pairing $\langle m',\, \epsilon^{\mathbf{B}^{+}}(\cD^{\mathbf{B}^{+}}) \rangle$ is defined via the lattice isomorphism $\Lambda_{O}^{\mathbf{B}^{+}} \simeq M_{\mathbf{P}'/\mathbf{H}'}$, $\chi_{f} \mapsto \chi_{f|_{F'}}$.
	Since the tori $P/H$ and $\mathbf{P}'/\mathbf{H}'$ are isomorphic via the assignment $pH \mapsto (g\alpha(p)g^{-1})\mathbf{H}'$, we have a lattice isomorphism $M_{P/H} \simeq M_{\mathbf{P}'/\mathbf{H}'}$ compatible with the isomorphisms $\Lambda_{O}^{\mathbf{B}^{+}} \simeq M_{\mathbf{P}'/\mathbf{H}'}$ and $\Lambda_{O}^{B^{+}} \simeq M_{P/H}$ and the inclusion $\Lambda_{O}^{\mathbf{B}^{+}} \subset \Lambda_{O}^{B^{+}}$ (here, the compatibility follows from the choice $g \in \mathbf{B}^{+}$).
    In particular, the equality $\Lambda_{O}^{\mathbf{B}^{+}} = \Lambda_{O}^{B^{+}}$ holds.
	Moreover, since $F \rightarrow F'=g.F$, $x \mapsto g.x$ is a toric isomorphism with respect to the isomorphism $P/H \simeq \mathbf{P}'/\mathbf{H}'$, the lattice isomorphism $M_{P/H} \simeq M_{\mathbf{P}'/\mathbf{H}'}$ preserves the Demazure roots, i.e., it sends $\cR_{P/H}(F)$ onto $\cR_{\mathbf{P}'/\mathbf{H}'}(F')$.
	Therefore in $\Lambda_{O}^{\mathbf{B}^{+}} = \Lambda_{O}^{B^{+}}$, $\cR^{+}_{G}(X)$, $\cS^{+}_{G}(X)$ and $\cU^{+}_{G}(X)$ are identified with $\cR^{+}_{\mathbf{G}}(X)$, $\cS^{+}_{\mathbf{G}}(X)$ and $\cU^{+}_{\mathbf{G}}(X)$, respectively.
    Furthermore, the $\mathbf{B}^{+}$-root subgroups associated to $\cR^{+}_{\mathbf{G}}(X)$ coincide with the $B^{+}$-root subgroups associated to $\cR^{+}_{G}(X)$,
    since they preserve the fibers of $\Phi$, and since the $B^{+}$-action factors through the $\mathbf{B}^{+}$-action (cf. Theorem~\ref{theorem:B-normalized Ga action}).
    In particular, $U^{+}_{m'}$ for $m' \in \cS^{+}_{\mathbf{G}}(X)$ are same with $U^{+}_{m}$ for $m \in \cS^{+}_{G}(X)$.
	Hence by Corollary~\ref{coro:reductive} and Theorem~\ref{thm:brion}, $\mathbf{G}$, $\Aut_{\mathbf{G}}(X)$ and $U^{+}_{m}$ for $m \in \cS^{+}_{G}(X)$ generate a Levi subgroup of $\Aut^{0}(X)$.
	Finally, since the $G$-action factors through the $\mathbf{G}$-action, we see that $\Aut_{\mathbf{G}}(X) \subset \Aut_{G}(X) \subset \mathbf{G}$, proving the statement.
\end{proof}

\section{Application: K-unstability of projective bundles over rational homogeneous spaces}\label{section:applications}

In this section, we apply the results of Section \ref{section:maintheorem} to the case of projective bundles defined by totally decomposable vector bundles over rational homogeneous spaces, which are well-known examples of toroidal horospherical varieties.
As a corollary, we construct Fano $\P^1$-bundles which are K-unstable.

\begin{theorem}\label{theorem:application}
	Let $Y$ be a rational homogeneous space and  let $L_{1},\,\ldots,\,L_{k}$ be line bundles on $Y$, with $k\geq 2$. 
	Set $X\coloneqq \P_{Y}(L_{1} \oplus \cdots \oplus L_{k})$.

    Let $G$ be the product of $(\C^{\times})^{k}$ and the universal cover of $\Aut^{0}(Y)$. 
Let $T$ be a maximal torus and $B^-$ be a Borel subgroup of $G$ so that $T\subset B^-\subset G$, and let $B^+$ be its opposite Borel subgroup. Let $P$ be the parabolic subgroup of $G$ containing $B^{-}$ such that $Y \simeq G/P$. For each $i$, let $\chi_{i}$ be the character of $P$ such that $L_{i} = G \times^{P} \C_{\chi_{i}}$.

Then $X$ is a toroidal horospherical $G$-variety such that the following hold:
\begin{itemize}
    \item $\cR^{+}_{G}(X)$ consists of $\chi_{i} - \chi_{j}$ ($i \not= j$) such that $L_{i} \otimes L_{j}^{\vee}$ is nef on $Y$.
    \item $\cS^{+}_{G}(X)$ consists of $\chi_{i} - \chi_{j}$ ($i \not= j$) such that $L_{i} \simeq L_{j}$ as line bundles on $Y$.
\end{itemize}
In particular, $\Aut^{0}(X)$ is reductive if and only if for any $i\not= j$ such that $L_{i}\not\simeq L_{j}$, $L_{i} \otimes L_{j}^{\vee}$ is not nef.
\end{theorem}

\begin{proof}

    Write $G \coloneqq \tilde{A} \times (\C^{\times})^{k}$ where $\tilde{A}$ denotes the universal cover of $\Aut^{0}(Y)$. 
	Recall that any line bundle on $Y$ is linearizable with respect to the natural $\tilde{A}$-action on $Y$.
	Thus we may assume that each $L_{i}$ is $\tilde{A}$-linearized.
	Any point $x \in X$ can be written as $x = [l_{1}\colon\cdots\colon l_{k}]$ where $l_{i} \in (L_{i})_{y}$ is a point on the fiber of $L_{i}$ over $y \in Y$. Consider the $G$-action defined as follows:
	\[
	(g,\, z_{1},\, \ldots,\, z_{k}).x = [g.(z_{1}l_{1}) \colon \cdots \colon g.(z_{k}l_{k})]
	\]
    for $g \in \tilde{A}$, and $(z_{1}, \ldots,\, z_{k}) \in (\C^{\times})^{k}$.
	
	We claim that $X$ is a toroidal horospherical $G$-variety and $Y$ is the associated rational homogeneous space.
	For any nonzero $l_{i} \in (L_{i})_{e P}$, the $G$-orbit of $[l_{1}\colon \cdots \colon l_{k}]$ is open in $X$.
	Furthermore, since $R^{u}(B^{-})$ acts trivially on each $(L_{i})_{e P}$, it fixes the point $[l_{1}\colon \cdots \colon l_{k}]$, and hence $X$ is horospherical.
	Moreover, since $P$ acts on $(L_{i})_{e P}$ via the character $\chi_{i}$, for any point $x = [l_{1} \colon \cdots \colon l_{k}] \in \P (L_{1} \oplus \cdots \oplus L_{k})_{e P}$ such that $l_{i} \not=0$ for all $i$, we have
	\[
	H \coloneqq \Stab_{G}(x) = \bigcap_{i \not= j} \ker (\chi_{i} - \chi_{j}),
	\]
	and hence $P = N_{G}(H)$ by Proposition~\ref{prop:normalizer-parabolic}.
	In other words, $Y$ is the rational homogeneous space associated to the horospherical variety $X$.
    It follows that $X$ is furthermore toroidal, since it admits a $G$-equivariant morphism onto $Y = G/P$.

	On the other hand, the fiber $F$ of the projection $X \rightarrow Y$ over $e P$ is the projective space $\P^{k-1}$, and the $P$-action is given as follows:
	\[
	p.[l_{1} \colon \cdots \colon l_{k}] = [\chi_{1}(p)l_{1} \colon \cdots \colon \chi_{k}(p)l_{k}].
	\]
	This action induces the $S$-toric variety structure on $F$, with $S=P/H$.
	Since Demazure roots of $F$ are the $S$-roots of $\Aut^{0}(F)\simeq \PGL_k$, we have
	\[
	\cR_{S}(F) = \{\chi_{i} - \chi_{j} \colon i \not=j\}.
	\]
	Therefore, by Lemma \ref{lemma: key lemma}(\ref{item:keylemma1}), $\cR^{+}_{G}(X)$ consists of $\chi_{i} - \chi_{j}$ ($i\not=j$) that are $B^{+}$-dominant.
	By the Borel--Weil--Bott theorem, it is equivalent to the nefness of the line bundle
	\[
	G \times^{P} \C_{\chi_{i} - \chi_{j}} \simeq L_{i} \otimes L_{j}^{\vee}.
	\]
	Thus the statement follows from Corollary \ref{coro:reductive}.
\end{proof}

If $k=2$ and $L_1=\cO_Y$, we get the following:

\begin{corollary}\label{coro:P1 bundle}
    Let $Y$ be a rational homogeneous space, let $L$ be a line bundle on $Y$ and  set $X\coloneqq \P_Y(\cO_Y\oplus L)$. Then $\Aut^0(X)$ is reductive if and only if either $L$ is trivial, or neither $L$ nor $L^{\vee}$ is nef on $Y$. 
    
    In particular, if $Y$ has Picard number $1$, $\Aut^0(X)$ is reductive if and only if $X\simeq Y\times\P^1$.
\end{corollary}

\begin{example}
    Let $Y=(\P^1)^n$, and consider the $\P^1$-bundle $X=\P_{Y}(\cO_Y\oplus L)$, where $L=\cO_Y(a_1,\ldots,a_n)$, with $a_i\in \Z$ for every $i=1,\ldots,n$.
    By Corollary \ref{coro:P1 bundle}, $\Aut^0(X)$ is reductive if and only if either $a_1=\ldots=a_n=0$, or there exists $i\neq j$ such that $a_ia_j<0$, for $i,j=1,\ldots,n$.
    On the other hand, $X$ is Fano if and only if $a_i\in \{-1,0,1\}$ for every $i=1,\ldots,n$ (see \cite[Theorem 6.4.9]{CLS}).
    
    While in dimension $2$ the only toric Fano variety with a $\P^1$-bundle structure and reductive automorphism group is $\P^1\times\P^1$, in higher dimension this is not true.
    For instance, by the previous paragraph, $X=\P_{\P^1\times\P^1}(\cO_{\P^1\times\P^1}\oplus \cO_{\P^1\times\P^1}(1,-1))$ is Fano and $\Aut^{0}(X)$ is reductive.
\end{example}

Let us recall that a Fano variety is called \emph{K-unstable} if it is not \emph{K-semistable}, where we refer to \cite{XuBook} for details on K-stability. For a horospherical variety $X$, $X$ is K-polystable if and only if $X$ is K-semistable (\cite[Corollary 5.7]{DelcroixENS2020}).  Recall that if $X$ is K-polystable then $\Aut^0(X)$ is reductive (see \cite{Matsushima}, \cite{KStabilityReductive}). 
Using Corollary \ref{coro:P1 bundle}, we shall give several examples of Fano $\P^{1}$-bundles that are K-unstable. To this end, let us recall that, given a Fano variety $Y$ and $L_1,\ldots,L_{k}$ nef line bundles on $Y$ such that $K_Y^\vee\otimes L_1^\vee \otimes\ldots L_k^\vee$ is ample, then the $\P^{k-1}$-bundle $X=\P_{Y}(\oplus_{i=1}^k L_i^\vee)$ is a Fano variety as well (see \cite[Example~3.3(2)]{Debarre}). In the special case of $k=2$ and $L_1=\cO_Y$, applying Corollary \ref{coro:P1 bundle} we get the following:

\begin{corollary}\label{coro:KUnstableCriterion}
    Let $Y$ be a rational homogeneous space, and let $L$ be a nontrivial nef line bundle such that $K_Y^{\vee}\otimes L^{\vee}$ is ample. Then $ X\coloneq \P_{Y}(\cO_Y\oplus L^{\vee})$ is a smooth K-unstable Fano variety.
\end{corollary}

Notice that certain $\P^1$-bundles over Fano varieties with Fano index $\ge 2$ have been proven to be K-unstable (see \cite[Theorem 1.1]{ZhangZhou}).
Here, recall that the \emph{Fano index} of a smooth Fano variety is the maximal integer $r$ such that $K_X^{\vee}\simeq L^{\otimes r}$, with $L\in \Pic(X)$.
With Corollary \ref{coro:KUnstableCriterion}, one may check K-unstability of $\P^{1}$-bundles over rational homogeneous spaces of Fano index 1.
In fact, even though no rational homogeneous space of Picard number 1 is of Fano index 1 (see \cite[Proposition, \S4.6]{AkhBook}), their products can have Fano index 1.
Let us conclude by presenting the simplest non-toric example:

\begin{example}\label{example:KunstableFanoIndex1}
    Let $Y=\P^1\times\Q^3$, where $\Q^3$ is the quadric threefold of $\P^4$, and choose $L=\cO_{\P^{1}}(1) \boxtimes \cO_{\Q^{3}}(1)$.
    Since $K_{Y}^{\vee} \simeq \cO_{\P^{1}}(2) \boxtimes \cO_{\Q^{3}}(3)$, the Fano index of $Y$ is the greatest common divisor of 2 and 3, i.e., 1.
    Furthermore, by Corollary \ref{coro:KUnstableCriterion}, $X\coloneqq \P_{Y}(\cO_Y\oplus L^{\vee})$ is a K-unstable Fano $\P^1$-bundle.
\end{example}


\end{document}